\documentclass[12pt]{article}
\usepackage{amsmath,amssymb,amsthm, amsfonts}
\usepackage{etaremune, enumerate, float, verbatim}
\usepackage{hyperref}
\usepackage{graphicx}
\usepackage{url}
\usepackage[mathlines]{lineno}
\usepackage{dsfont} 
\usepackage{tikz, graphicx, subcaption, caption}
\usepackage[algo2e,ruled,vlined]{algorithm2e}
\usepackage{mathrsfs,amsmath}
\usepackage{color}
\definecolor{red}{rgb}{1,0,0}

\definecolor{blue}{rgb}{0,0,.7}

\definecolor{green}{rgb}{0,.6,0}

\definecolor{purp}{rgb}{.5,0,.5}

\numberwithin{figure}{section}   

\newtheorem{thm}{Theorem}[section]

\newtheorem{quest}[thm]{Question}

\newcommand{\be}{{\bf e}}

\newcommand{\osl}{\mathcal{S}}
\theoremstyle{definition}

\newcommand{\bone}{{\bf \mathds{1}}}
\theoremstyle{definition}

\theoremstyle{definition}

\newcommand{\e}{\operatorname{E}}
\newcommand{\ptw}{\operatorname{ptw}}
\newcommand{\Z}{\operatorname{Z}}

\newcommand{\pt}{\operatorname{pt}}

\newcommand{\eptw}{\operatorname{eptw}}
\newcommand{\cptw}{\operatorname{cptw}}

\newcommand{\bit}{\begin{itemize}}
\newcommand{\eit}{\end{itemize}}
\newcommand{\ben}{\begin{enumerate}}
\newcommand{\een}{\end{enumerate}}
\newcommand{\beq}{\begin{equation}}
\newcommand{\eeq}{\end{equation}}
\newcommand{\bea}{\begin{eqnarray*}} 
\newcommand{\eea}{\end{eqnarray*}}
\newcommand{\bpf}{\begin{proof}}
\newcommand{\epf}{\end{proof}\ms}
\newcommand{\bmt}{\begin{bmatrix}}
\newcommand{\emt}{\end{bmatrix}}
\newcommand{\ms}{\medskip}

\newcommand{\lp}{\!(}
\newcommand{\rp}{)}

\title{Propagation time for weighted zero forcing}
\author{P.A. CrowdMath}

\begin{document}
\maketitle


\begin{abstract}
Zero forcing is a graph coloring process that was defined as a tool for bounding the minimum rank and maximum nullity of a graph. It has also been used for studying control of quantum systems and monitoring electrical power networks. One of the problems from the 2017 AIM workshop \emph{Zero forcing and its applications} was to explore edge-weighted probabilistic zero forcing, where edges have weights that determine the probability of a successful force if forcing is possible under the standard zero forcing coloring rule.

In this paper, we investigate the expected time to complete the weighted zero forcing coloring process, known as the expected propagation time, as well as the time for the process to be completed with probability at least $\alpha$, known as the $\alpha$-confidence propagation time. We demonstrate how to find the expected and confidence propagation times of any edge-weighted graph using Markov matrices. We also determine the expected and confidence propagation times for various families of edge-weighted graphs including complete graphs, stars, paths, and cycles. 
\end{abstract}

\section{Introduction}\label{s:intro}
Zero forcing is a graph coloring process that was defined in \cite{AIM08} as a tool for bounding the minimum rank and maximum nullity of a graph, and it was later used for studying control of quantum systems \cite{BG, S}. Zero forcing is also used in another graph coloring process called power domination, which is a graph theoretic model of the phase measurement unit (PMU) placement problem in electrical power networks. Zero forcing has also been applied to graph searching \cite{Y}. 

Zero forcing is defined in terms of a color change rule, where there is an initial set of vertices that are colored blue, with all other vertices in the graph colored white. If $u$ is a blue vertex that has exactly one white neighbor $v$, then the color of $v$ is changed to blue on the next step. With zero forcing, researchers have investigated the minimum possible size of a set of vertices that can be used to color the whole graph $G$ with repeated iterations of the color change rule (the \emph{zero forcing number} $\Z(G)$), and the minimum time that it takes for the graph to be colored by a minimum zero forcing set (the \emph{propagation time} $\pt(G)$) \cite{proptime, PSDpropTime, BY13}.

Many variants of zero forcing have been studied with different color change rules, each with their own variation of propagation time, including positive semidefinite zero forcing \cite{smallparam, PSDthrottle, trth}, skew zero forcing \cite{IMA, King15, mr0, skewth}, and probabilistic zero forcing \cite{KY13, GH18-PZF, pzf2}. In this paper, we investigate a new variant of zero forcing that was proposed at the 2017 AIM workshop \emph{Zero forcing and its applications} \cite{aimzf}. At the AIM workshop, they asked the following question.

\begin{quest}
What can we say about edge-weighted probabilistic zero forcing, i.e. edges have weights and if standard zero forcing says we can force we do so with probability of the edge weight?
\end{quest}

We call the resulting graph coloring process \emph{weighted zero forcing}. Specifically we start with an edge-weighted graph $G$ with edge weights in $(0, 1)$ and an initial set $B$ of blue vertices. Whenever a blue vertex $u$ can color a white vertex $v$ using the standard zero forcing rule, we do so with probability equal to the weight of the edge between $u$ and $v$. If the whole graph $G$ can eventually be colored blue with weighted zero forcing starting with only the vertices of $B$ blue, we say that $B$ is a \emph{weighted zero forcing set} of $G$. The \emph{weighted zero forcing number} of $G$ is the minimum possible size of a weighted zero forcing set of $G$. 

Note that the weighted zero forcing number of $G$ is just the standard zero forcing number of the unweighted graph $G'$ obtained from $G$ by removing the edge weights, so we use the same notation $\Z(G)$ as the standard zero forcing number to denote the weighted zero forcing number of $G$. Also observe that the weighted zero forcing sets of $G$ are the zero forcing sets of $G'$.

The {\em propagation time} of a nonempty set  $B$ of vertices of an edge-weighted graph $G$, $\ptw(G,B)$, is a random variable that reflects the time (number of the round) at which the last white vertex turns blue when applying a weighted zero forcing process starting with the set $B$ blue. For an edge-weighted graph $G$ and a set $B\subseteq V(G)$ of vertices, the {\em expected propagation time of $B$ for $G$}  is the expected value of the propagation time of $B$, i.e., 
\[\eptw(G,B)=\e [ \ptw(G,B)].\]
  The {\em expected propagation time} of an edge-weighted graph $G$ is the minimum of the expected propagation time of $B$ for $G$ over all minimum weighted zero forcing sets $B$ of $G$, i.e., 
\[\eptw(G)=\min\{\eptw(G,B):|B| = \Z(G)\}.\]

We also define the {\em $\alpha$-confidence propagation time of $B$ for $G$}, denoted $\cptw(G, B, \alpha)$, as the minimum $t$ for which the probability is at least $\alpha$ that $G$ is fully colored after $t$ steps starting with the vertices of $B$ blue. The {\em $\alpha$-confidence propagation time of $G$}, denoted $\cptw(G, \alpha)$, is the minimum of $\cptw(G, B, \alpha)$ over all minimum weighted zero forcing sets $B$ of $G$, i.e.,
\[\cptw(G,\alpha)=\min\{\cptw(G,B,\alpha):|B| = \Z(G)\}.\]

In Section \ref{s:gen} we explain how to find $\eptw(G)$ for any edge-weighted graph $G$ using Markov matrices. We also explain how to find $\cptw(G, \alpha)$. In Section \ref{s:ept} we determine $\eptw(G)$ exactly when $G = K_{1, n}$ (a star of order $n+1$), $G = K_n$ (a complete graph of order $n$), $G=P_n$ (a path of order $n$), $G=C_n$ (a cycle of order $n$). In Section \ref{s:cpt} we determine $\cptw(G, \alpha)$ for the same families of graphs. In Section \ref{s:open}, we discuss some further directions for research on weighted zero forcing. 

\section{Markov matrices for exact values of propagation times for edge-weighted graphs}\label{s:gen}
Given an initial set $B$ of blue vertices, the weighted zero forcing coloring process of an edge-weighted graph $G$ is a Markov chain. Each possible state of the Markov chain represents a possible set of blue vertices. Each state corresponds to some subset $C \subset V(G)$ for which $B \subset C$, so there are a total of $s = 2^{|V(G)|-|B|}$ states. 

More precisely we say that a {\em properly ordered state list for $B$}, denoted by  $\osl=(S_1,\dots, S_s)$, is an ordered list of all  states  for $B$ in which $S_1$ is the initial state (where exactly the vertices in $B$ are blue), $S_s$ is the final state (where all vertices are blue), and  $|S_i|<|S_j|$ implies  $i<j$. 

We construct an $s \times s$ matrix $M$ where the $(i, j)$ entry is the probability of transitioning from state $S_i$ to state $S_j$ in one time step. This matrix is upper-triangular since the cardinality of states is non-decreasing. The probability that all vertices are blue after round $r$ is $\lp M^r\rp_{1s}$, so we obtain a formula for the expected propagation time of $B$ like the formula for probabilistic zero forcing in \cite{GH18-PZF}.

\[\eptw(G,B)=\sum_{r=1}^\infty r\left(\lp M^r\rp_{1s}-\lp M^{r-1}\rp_{1s}\right).\]

As with probabilistic zero forcing in \cite{pzf2}, it is possible to obtain a simpler formula for $\eptw(G, B)$. 

\begin{thm}\label{t:ept-M} Suppose that $G$ is a graph,  $ B\subset V(G)$ is nonempty,  $\osl$ is a properly ordered state list for $B$ with $s$ states, and  $M=M(G,\osl)$. Then \[\eptw(G,B)=((M-\bone{\be_s}^T-I)^{-1})_{1s}+1,\]
where $\bone=[1,\dots,1]^T$ and $\be_s=[0,\dots,0,1]^T$. 
\end{thm}

Given the formulas for $\eptw(G, B)$ we can determine $\eptw(G)$ for any graph $G$ by calculating the minimum of $\eptw(G, B)$ over all minimum weighted zero forcing sets $B$ of $G$. For $\cptw(G, \alpha)$, we can compute powers $M^r$ for each minimum weighted zero forcing set $B$ until we find an $r$ such that $\lp M^r\rp_{1s} \geq \alpha$. The first such $r$ among all of the minimum zero forcing sets $B$ gives the value of $\cptw(G, \alpha)$.

\section{Expected propagation times for families of edge-weighted graphs}\label{s:ept}

In this section, we calculate the expected propagation times of several families of edge-weighted graphs. These include complete graphs, complete bipartite graphs, paths, and cycles. We start with complete graphs.

\begin{thm}
If $G = K_n$ with edge weights $w_{\left\{i, j\right\}} \in (0, 1)$, then $\eptw(G) = \tfrac{1}{1-W}$ where $W$ is the minimum possible value of $\prod_{j \neq i} (1-w_{\left\{i, j\right\}})$ over all $1 \leq i \leq n$.
\end{thm}

\begin{proof}
Let $Z$ be the minimum forcing set consisting of vertices $v_1,v_2,\dots,v_{n-1}$ and let $w_{i, n}$ be the weight of the edge from $v_i$ to $v_n$. As only $v_n$ remains to be forced, $\eptw(G,Z)$ is the reciprocal of the probability that $v_n$ is forced during the first step or $\tfrac{1}{1-W}$ where $W=\prod_{j \neq n} (1-w_{\left\{n, j\right\}})$.
\end{proof}

Next we find the expected propagation time of stars with edge weights, before finding the expected propagation time of complete bipartite graphs.

\begin{thm}
If $G = K_{1,n}$ with edge weights $w_1, \dots, w_n \in (0, 1)$, then $\eptw(G)$ is the minimum over $1 \leq i \leq n$ of $\tfrac{1}{w_i} + \tfrac{1}{1-W}$, where $W = \prod_{j \neq i} (1-w_j)$. 
\end{thm}

\begin{proof}
For the star graph $K_{1,n}$ let $v_1,v_2,\dots,v_n$ be the outer vertices with weights $w_1,w_2,\dots,w_n$ towards the center $v_0$. Let $Z_1$ be the minimum zero forcing set consisting of the first $n-1$ outside vertices and let $Z_2=Z_1\cup v_0$. For simplicity, let $W= \prod_{j =1}^{n-1} (1-w_j)$ be the probability that $v_0$ is not forced during the first time step. Then $\eptw(G, Z_1) =W \cdot \eptw(G, Z_1) + (1-W) \cdot \eptw(G , Z_2)+1$. Also $\eptw(G,Z_2) = \tfrac{1}{w_n}$ since the only remaining white vertex is $v_n$ which is forced with probability $w_n$. Therefore $\eptw(G) = \tfrac{1}{w_n} + \tfrac{1}{1-W}$.
\end{proof}

\begin{thm}
If $G = K_{a,b}$ for $a, b \ge 2$ with parts $A_1, A_2, \dots, A_a$ and $B_1, B_2, \dots, B_b$ such that $p_{i,j}$ denotes the weight of edge $A_{i}B_{j}$, then $\eptw(G)$ is the minimum over $1 \leq x \leq a$ and $1 \leq y \leq b$ of $\frac{1}{1 - p_x} + \frac{1}{1 - p_y}  - \frac{1}{1-p_x p_y}$, where  $p_x = (1-p_{x,1})(1-p_{x,2})\dots(1-p_{x,b})$ and $p_y = (1-p_{1,y})(1-p_{2,y})\dots(1-p_{a,y})$.
\end{thm}

\begin{proof}
Suppose that $A_{x}$ and $B_{y}$ are the only white vertices. Then the probability that $A_x$ is not colored on any given step is $p_x = (1-p_{x,1})(1-p_{x,2})\dots(1-p_{x,b})$ and the probability that $B_y$ is not colored on any given step is $p_y = (1-p_{1,y})(1-p_{2,y})\dots(1-p_{a,y})$. If $T$ is the random variable for the number of steps in the coloring process, we obtain $\mathbb{E}(T) = \sum _{k = 1}^{\infty} P(T \geq k) = \sum _{k = 1}^{\infty} (1 - P(T < k)) = \sum_{k = 1}^{\infty} (1 - (1-p_x^{k-1})(1-p_y^{k-1})) = \frac{1}{1 - p_x} + \frac{1}{1 - p_y}  - \frac{1}{1-p_x p_y}.$ Then we pick $x$ and $y$ to minimize this quantity.
\end{proof}

In the next result, we find the expected propagation time of edge-weighted paths, before using a similar method for edge-weighted cycles.

\begin{thm}
If $G = P_n$ with edge weights $w_1, \dots, w_{n-1}$ in order, then $\eptw(G) = \sum_{k=1}^{n-1} \frac{1}{w_k}$.
\end{thm}

\begin{proof}
The minimum forcing set of a path is a single vertex located at one of its endpoints. Let $Z_k$ denote the forcing set of a path consisting of the first $k$ vertices in the forcing chain of one of the endpoints of $P_n$. Given $k$ forced vertices, the probability that the $(k+1)^{st}$ vertex is forced in the next time step is $w_k$. Therefore $\eptw(P_n,Z_k)=w_k \cdot \eptw(P_n,Z_{k+1}) + (1-w_k) \cdot \eptw(P_n,Z_k)+1$ so $\eptw(P_n,Z_k) = \eptw(P_n,Z_{k+1}) + \tfrac{1}{w_k} $ for all $1 \le k \le n-1$. Since $\eptw(P_n,Z_n) = 0$, we obtain $\eptw(P_n) = \eptw(P_n,Z_1) = \sum_{k=1}^{n-1} \frac{1}{w_k}$.
\end{proof}

Observe that the expected propagation time of cycles can be calculated recursively in a similar manner. Consider the graph $C_3$ with weights $w_1,w_2,w_3$ for edges $\left\{1,2\right\},\left\{2,3\right\},\left\{3,1\right\}$ respectively. Suppose that our minimum zero forcing set consists of the vertices $1,3$. Let $Z_{a,b}^3$ denote the zero forcing set consisting of $a$ consecutive vertices in the forcing chain of vertex $1$ and $b$ consecutive vertices in the forcing chain of vertex $3$. Furthermore, let $\eptw(C_3,Z^3_{a,b})=0$ whenever $a+b \ge 3$ since all the vertices would already be forced.

Then $\eptw(C_3,Z^3_{1,1})=w_1(1-w_2) \cdot \eptw(C_3,Z^3_{2,1})+(1-w_1)w_2 \cdot \eptw(C_3,Z^3_{1,2})+(1-w_1)(1-w_2) \cdot \eptw(C_3,Z^3_{1,1})+1$ so our expected propagation for $C_3$ is $\eptw(C_3,Z^3_{1,1}) = \tfrac{1}{w_1+w_2-w_1w_2}$ since $\eptw(C_3,Z^3_{2,1})=\eptw(C_3,Z^3_{1,2})=0$. Now define $E_3(w_1,w_2) =\tfrac{1}{w_1+w_2-w_1w_2}$.

Using analogous notation for $C_4$, we have\begin{align*} \eptw(C_4,Z^4_{1,1}) &=w_1(1-w_3) \cdot \eptw(C_4,Z^4_{2,1})+(1-w_1)w_3 \cdot \eptw(C_4,Z^4_{1,2}) \\&\quad +(1-w_1)(1-w_3) \cdot \eptw(C_4,Z^4_{1,1})+w_1w_3 \cdot \eptw(C_4,Z^4_{2,2}) +1 \end{align*}

Notice that $\eptw(C_4, Z^4_{2,1})=E_3(w_2,w_3)$ since we only have a single vertex $3$ left to force, a state identical to that of $C_3$. Similarly $\eptw(C_4, Z^3_{1,2})=E_3(w_1,w_2)$, so $\eptw(C_4,Z^4_{1,1}) = \tfrac{1}{w_1+w_3-w_1w_3}(w_1(1-w_3) \cdot E_3(w_2,w_3) + (1-w_1)w_3 \cdot E_3(w_1,w_2) +1)$. 

In general, if we define $E_n(w_1,w_2,\dots, w_{n-1})=\eptw(C_n,Z^n_{1,1})$, the expected propagation time of a $C_n$ with consecutive weights $w_1,w_2,\dots,w_{n-1}$ along the forcing path, we have $\eptw(C_n,Z^n_{1,1})=w_1(1-w_{n-1}) \cdot \eptw(C_n,Z^n_{2,1}) + (1-w_1)w_{n-1} \cdot \eptw(C_n,Z^n_{1,2})+(1-w_1)(1-w_{n-1}) \cdot \eptw(C_n,Z^n_{1,1})+w_1 w_{n-1} \cdot \eptw(C_n,Z^n_{2,2})+1 = \tfrac{1}{w_1+w_{n-1} - w_1 w_{n-1}} ( w_1(1-w_{n-1}) \cdot \eptw(C_n,Z^n_{2,1}) +
(1-w_1)w_{n-1} \cdot \eptw(C_n,Z^n_{1,2})+w_1 w_{n-1} \cdot \eptw(C_n,Z^n_{2,2})  +1 ) =$\\
$\tfrac{1}{w_1+w_{n-1} - w_1 w_{n-1}} ( w_1(1-w_{n-1}) \cdot E_{n-1} (w_2,w_3,\dots,w_{n-1}) + (1-w_1)w_{n-1} \cdot E_{n-1}(w_1,w_2,\cdots,w_{n-2})+ w_1w_{n-1} \cdot E_{n-2}(w_2,w_3,\dots,w_{n-2})  +1 )$.

With the recurrence, we can find $\eptw(C_n,Z^n_{1,1})$. To find $\eptw(C_n)$ simply take the minimum over all cyclic placements of $Z^n_{1,1}$.

\section{Confidence propagation time for families of edge-weighted graphs}\label{s:cpt}

In this section, we determine confidence propagation times for special families of graphs including complete graphs, stars, paths, and cycles. We start with complete graphs and stars.

\begin{thm}
If $G = K_n$ with edge weights $w_{\left\{i, j\right\}} \in (0, 1)$, then $\cptw(G,\alpha)$ is $\lceil \log _{P_0} (1 - \alpha) \rceil$, where $P_0$ is the minimum of $\prod_{j \neq i} (1-w_{\left\{i, j\right\}})$ over all $1 \leq i \leq n$.
\end{thm}

\begin{proof}
Suppose $v_i$ is the white vertex. Then the probability that $v_i$ is colored at any given moment is $1 -\prod_{j \neq i} (1-w_{\left\{i, j\right\}}) = 1 - P$, so the probability that it takes $m$ steps or less is the complement of the probability that it takes more than $m$ steps, which is just $1 - P^m$. Then minimize $P$ over all choices of $v_i$, say at $P_0$, and solve for the minimum $m$ with $1 - P^m \geq \alpha$, i.e., $m = \lceil \log _{P_0} (1 - \alpha) \rceil$.
\end{proof}

\begin{thm}
If $G = K_{1,n}$ with edge weights $p_1 \leq p_2 \leq \dots \leq p_n$, then $\cptw(G, \alpha)$ is the minimum value of $m$ for which $p_n (1-P) \frac{\frac{1 - P^m}{1-P}  - \frac{1 - (1-p_n)^m}{p_n}}{P - (1 - p_n)} \geq \alpha$.
\end{thm}

\begin{proof}
We first find the probability that the coloring process completes in exactly $s$ steps if we initially color all vertices except the center and vertex $i$. We let $P = \prod_{j \neq i} (1-p_j)$. At some point, we must have two forces, which have probabilities $1 - P$ and $p_i$ of occurring. We pick some number of the remaining $s-2$ steps to be non-forces before coloring the center, so our total probability is
$$p_i(1-P) \cdot \sum_{j=0}^{s-2} P^j (1-p_i)^{s-2-j} = p_i (1-P) \cdot \frac{P^{s-1} - (1-p_i)^{s-1}}{P - (1 - p_i)},$$which is maximized when $p_i$ is maximized, or $i=n$.

Then the probability that $K_{1,n}$ is all blue after at most $m$ steps is
$$p_n (1-P) \cdot \sum_{s=2}^{m} \frac{P^{s-1} - (1-p_n)^{s-1}}{P - (1 - p_n)} = p_n (1-P) \frac{\frac{1 - P^m}{1-P}  - \frac{1 - (1-p_n)^m}{p_n}}{P - (1 - p_n)}.$$
\end{proof}

In the next two results, we evaluate confidence propagation times for paths and cycles where all of the edge weights are equal. We use generating functions in both proofs.

\begin{thm}
If $G = P_n$ with edge weights all equal to $p \in (0, 1)$, then $\cptw(G, \alpha)$ is the minimum value of $m$ for which $\sum_{s=n-1}^{m} \binom{m}{s}p^s(1-p)^{m-s} \geq \alpha$. 
\end{thm}

\begin{proof}
The weighted zero forcing of $G$ is equivalent to the following process. We start with the number $1$. Each step, we add $1$ with probability $p$, and $0$ with probability $1-p$. We will determine the probability that we reach $n$ in $m$ steps or less.

Consider the generating function $(px + (1-p))^m$, where the coefficient of $x^s$ is the probability that the number is $s+1$ after $m$ steps. The coefficient of $x^s$ is $\binom{m}{s}p^s(1-p)^{m-s}$, and the probability in question is
$$\sum_{s=n-1}^{m} \binom{m}{s}p^s(1-p)^{m-s}$$. Then $pt(P_n, \alpha)$ is the least $m$ for which the sum is at least $\alpha$.
\end{proof}

\begin{thm}
If $G = C_n$ with edge weights all equal to $p \in (0, 1)$, then $\cptw(G, \alpha)$ is the minimum value of $m$ for which $\sum _{s = n-2} ^{m} \binom{2m}{s}p^s(1-p)^{2m-s} \geq \alpha$. 
\end{thm}

\begin{proof}
A minimum zero forcing set of $G$ is two adjacent vertices, so we can consider the equivalent coloring process on $P_n$ with both endpoints blue. We start with the number $2$. Each step, we add $0$, $1$, or $2$ with probability $(1-p)^2, 2p(1-p), p^2$ respectively. We will determine the probability that we reach $n$ in $m$ steps or less.

Then we proceed as usual: consider the generating function $(p^2x^2 + 2p(1-p)x + (1-p)^2)^m= (px + (1-p))^{2m}$. The coefficient of $x^s$, or
$$\binom{2m}{s}p^s(1-p)^{2m-s}$$is the probability that our number is $s+2$ after $m$ steps. So the probability that we reach $n$ in $m$ steps or less is
$$\sum _{s = n-2} ^{m} \binom{2m}{s}p^s(1-p)^{2m-s}.$$
\end{proof}

In the next result, we evaluate confidence propagation times for paths with any distinct edge weights in $(0, 1)$.

\begin{thm}
If $G = P_{n+1}$ with distinct edge weights $w_1, w_2, \dots, w_n \in (0, 1)$ in path order, then $\cptw(G, \left\{1\right\},\alpha)$ is the minimum value of $t$ with $w_1 w_2 \dots w_n \sum_{i=1}^{n} \dfrac{(1-w_i)^{n-1} - (1 - w_i)^{t}}{w_i W_i}\geq \alpha$, where $W_i = \prod_{j \neq i}(w_j - w_i)$.
\end{thm}

\begin{proof}
Consider a path $P_{n+1}$ with $n+1$ vertices, and $n$ edges with weights $w_1, w_2, \dots, w_n$ in path order. Suppose that the first blue vertex is the endpoint of the $w_1$ edge. Let $W_i = \prod_{j \neq i}(w_j - w_i)$ for $1 \leq i \leq n$, let $A_{s, t}$ denote the set of $s$-tuples of non-negative integers that add to $t$, and for any $a \in A_{s, t}$ let $a_i$ denote the $i^{th}$ term in $a$. Then we claim that the probability the entire graph is blue after $t$ turns is
$$w_1 w_2 \dots w_n \left [ \sum_{i=1}^{n} \dfrac{(1-w_i)^{n-1} - (1 - w_i)^{t}}{w_i W_i} \right ].$$ 

To prove this, we use the fact that if $X_{n, i} =\prod_{j \neq i} (x_i - x_j)$, then $\displaystyle \sum_{a \in A_{n, k}} x_1^{a_1} x_2^{a_2} \dots x_n^{a_n} = \displaystyle \sum_{i=1}^n \frac{x_i^{k+n-1}}{X_{n, i}}$ for all integers $n \geq 2$, $k \geq 1-n$. Note that for the entire graph to be blue, there must have been at least $n$ forces, each with probability $w_1, w_2, \dots , w_n$, and for the remaining $t-n$ turns, there must have been no forces, and these non-forcing turns had probabilities of $1-w_1, 1-w_2, \dots, 1-w_n$.

Thus, the probability that the whole graph is blue after exactly $t$ turns is
$$w_1 w_2 \dots w_n \sum \limits_{a \in A_{n, t-n}} (1-w_1)^{a_1} (1-w_2)^{a_2} \dots (1-w_n)^{a_n}.$$Using the identity, this sum reduces to
$$w_1 w_2 \dots w_n \sum \limits_{i=1}^n \frac{(1-w_i)^{t-1}}{W_i}.$$ So the probability that all vertices in the graph are blue after at most $t$ turns is
$$w_1 w_2 \dots w_n \sum \limits_{s=n}^t \sum \limits_{i=1}^n \frac{(1-w_i)^{s-1}}{W_i} = w_1 w_2 \dots w_n \left [ \sum_{i=1}^{n} \dfrac{(1-w_i)^{n-1} - (1 - w_i)^{t}}{w_i W_i} \right ].$$

Thus, $\cptw(G, \alpha)$ is the minimum value of $t$ such that the last sum is at least $\alpha$.
\end{proof}

\section{Concluding remarks and further directions}\label{s:open}

In this paper, we defined propagation time for weighted zero forcing and found formulas using Markov matrices for the expected propagation times and confidence propagation times of all edge-weighted graphs with edge weights in $(0, 1)$. For special families of graphs including complete graphs, complete bipartite graphs, paths, and cycles, we found much simpler formulas for both propagation times. It would be interesting to investigate these quantities for other special families of graphs such as ladders, complete $k$-partite graphs, hypercubes, grids, spiders, caterpillars, and complete binary trees.

Besides zero forcing number and propagation time, another related parameter of graphs is the throttling number, which is the minimum of $k+\pt_k(G)$ over all positive integers $k$, where $\pt_k(G)$ is the minimum possible propagation time of $G$ with standard zero forcing using $k$ initial blue vertices. Throttling has been studied for standard zero forcing \cite{BY13}, positive semidefinite zero forcing \cite{PSDthrottle, trth}, skew zero forcing \cite{skewth}, probabilistic zero forcing \cite{GH18-PZF, NS}, and pursuit evasion games \cite{CRthrottle, CRthrottle2, gthr}. It is natural to define and investigate throttling for weighted zero forcing, both for expected propagation time and confidence propagation time.

Some other directions for future research are to define weighted positive semidefinite zero forcing, weighted skew zero forcing, and variants of probabilistic zero forcing that depend on edge weights. For each of the variants of edge-weighted zero forcing, it would be interesting to investigate properties and applications of their Markov matrices.

\section{Acknowledgments}

CrowdMath is an open program created by the MIT Program for Research in Math, Engineering, and Science (PRIMES) and Art of Problem Solving that gives high school and college students all over the world the opportunity to collaborate on a research project. The 2019 CrowdMath project is online at http://www.artofproblemsolving.com/polymath/mitprimes2019.

\end{document}